\documentclass[a4paper,12pt]{article}
\usepackage[utf8]{inputenc}
\usepackage[IL2]{fontenc}
\usepackage{a4wide,amsmath,amssymb,amsthm,tikz,amscd,hyperref}

\author{Zuzana Kr\v{c}m\'arikov\'a$^{1,}$\thanks{Supported by the Czech Science Foundation, grant No.\ 13-03538S, and by the Grant Agency of the Czech Technical University in Prague, grant No.\ SGS14/205/OHK4/3T/14}, Wolfgang Steiner$^{2,}$\thanks{Supported by the ANR-FWF project “Fractals and Numeration” (ANR-12-IS01-0002, FWF I1136) and the ANR project ``Dyna3S'' (ANR-13-BS02-0003)}, Tom\'a\v{s} V\'avra$^{1,}$\footnotemark[1]\\[4mm]
\emph{$^1$ Department of Mathematics FNSPE}\\
\emph{ Czech Technical University in Prague}\\
\emph{Trojanova 13, 120 00 Praha 2, Czech Republic}\\
\texttt{zuzka.krcmarikova@gmail.com, vavrato@gmail.com}\\[2mm]
 \emph{$^2$ IRIF, CNRS UMR 8243, Universit\'e Paris Diderot -- Paris 7}\\
 \emph{Case 7014, 75205 Paris Cedex 13, France}\\
\texttt{steiner@irif.fr}\\
}
\title{Finite beta-expansions with negative bases}
\date{}

\newtheorem{theorem}{Theorem}
\newtheorem{lemma}[theorem]{Lemma}
\newtheorem{proposition}[theorem]{Proposition}
\newtheorem{remark}[theorem]{Remark}

\begin{document}
\maketitle
\begin{abstract}
The finiteness property is an important arithmetical property of beta-expansions.
We exhibit classes of Pisot numbers $\beta$ having the negative finiteness property, that is the set of finite $(-\beta)$-expansions is equal to $\mathbb{Z}[\beta^{-1}]$.
For a class of numbers including the Tribonacci number, we compute the maximal length of the fractional parts arising in the addition and subtraction of $(-\beta)$-integers.
We also give conditions excluding the negative finiteness property.
\end{abstract}

\section{Introduction}
Digital expansions in real bases $\beta > 1$ were introduced by R\'enyi~\cite{REN}.
Of particular interest are bases~$\beta$ satisfying the \emph{finiteness property}, or Property~(F), which means that each element of $\mathbb{Z}[\beta^{-1}] \cap [0,\infty)$ has a finite (greedy) \emph{$\beta$-expansion}.
We know from Frougny and Solomyak~\cite{FRSO} that each base with Property~(F) is a Pisot number, but the converse is not true.
Partial characterizations are due to~\cite{FRSO,HOL,AKI}.
In~\cite{ABBPT}, Akiyama et al.\ exhibited an intimate connection to \emph{shift radix systems} (SRS), following ideas of Hollander~\cite{HOL}.
For results on shift radix systems (with the finiteness property), we refer to the survey~\cite{KS14}.

Numeration systems with negative base $-\beta<-1$, or \emph{$(-\beta)$-expansions}, received considerable attention since the paper~\cite{IS09} of Ito and Sadahiro in~2009.
They are given by the \emph{$(-\beta)$-transformation}
\[
T_{-\beta}:\, [\ell_\beta, \ell_\beta+1) \to [\ell_\beta, \ell_\beta+1), \quad x \mapsto -\beta x-\lfloor-\beta x-\ell_\beta\rfloor, \quad \mbox{with}\ \ell_\beta = \tfrac{-\beta}{\beta+1};
\]
see Section~\ref{sec:preliminaries} for details.
Certain arithmetic aspects seem to be analogous to those for positive base systems \cite{FRLAI11, MPV11}, others are different, e.g., both negative and positive numbers have $(-\beta)$-expansions; for $\beta < \frac{1+\sqrt{5}}{2}$, the only number with finite $(-\beta)$-expansion is~$0$.
We say that $\beta > 1$ has the \emph{negative finiteness property}, or Property~($-$F), if each element of $\mathbb{Z}[\beta^{-1}]$ has a finite $(-\beta)$-expansion.
By Dammak and Hbaib~\cite{DH}, we know that $\beta$ must be a Pisot number, as in the positive case.
It was shown in~\cite{MPV11} that the Pisot roots of $x^2-mx+n$, with positive integers $m, n$, $m\geq n+2$, satisfy the Property~($-$F).
This gives a complete characterization for quadratic numbers, as $\beta$ does not possess Property ($-$F) if $\beta$ has a negative Galois conjugate, by~\cite{MPV11}.

First, we give other simple criteria when $\beta$ does \emph{not} satisfy Property~($-$F).
Surprisingly, this happens when $\ell_\beta$ has a finite $(-\beta)$-expansion, which is somewhat opposite to the positive case, where Property~(F) implies that $\beta$ is a simple Parry number.

\begin{theorem} \label{thm:notf}
If $T_{-\beta}^k(\ell_\beta) = 0$ for some $k \ge 1$, or if $\beta$ is the root of a polynomial $p(x) \in \mathbb{Z}[x]$ with $|p(-1)|=1$, then $\beta$ does not possess Property ($-$F).
\end{theorem}

The main tool we use is a generalization of shift radix systems. We show that the $(-\beta)$-transformation is conjugated to a certain $\alpha$-SRS. Then we study properties of this dynamical system. We obtain a complete characterization for cubic Pisot units.

\begin{theorem} \label{t:cubic}
Let $\beta>1$ be a cubic Pisot unit with minimal polynomial $x^3-ax^2+bx-c$.
Then $\beta$ has Property ($-$F) if and only if $c=1$ and $-1\le b<a$, $|a|+|b| \ge 2$.
\end{theorem}

Considering Pisot numbers of arbitrary degree, we have the following results.

\begin{theorem} \label{t:mdnacci}
Let $\beta>1$ be a root of $x^{d}-mx^{d-1}-\dots-mx-m$ for some positive
integers~$d,m$.
Then $\beta$ has Property ($-$F) if and only if $d\in\{1,3,5\}$.
\end{theorem}

\begin{theorem} \label{t:dominant}
Let $\beta>1$ be a root of $x^d-a_{1}x^{d-1}+a_2x^{d-2}+\dots + (-1)^d a_d\in\mathbb{Z}[x]$ with $ a_i\ge 0$ for $i=1,\dots,d$, and $a_{1}\ge 2+\sum_{i=2}^{d} a_{i}.$ Then $\beta$ has Property ($-$F).
\end{theorem}

These theorems are proved in Section~\ref{sec:finiteness--beta}.
In Section~\ref{sec:addition-subtraction}, we give a precise bound on the number of fractional digits arising from addition and subtraction of $(-\beta)$-integers in case $\beta>1$ is a root of $x^3-mx^2-mx-m$ for $m\geq1$.
This is based on an extension of shift radix systems.
The corresponding numbers for $\beta$-integers have not been calculated yet, although they can be determined in a similar way.

\section{($-\beta$)-expansions} \label{sec:preliminaries}
For $\beta>1$, any $x\in[\ell_\beta,\ell_\beta+1)$ has an expansion of the form
\begin{displaymath}
x = \sum_{i=1}^\infty \frac{x_i}{(-\beta)^i} \quad \mbox{with} \quad
x_i=\lfloor-\beta T_{-\beta}^{i-1}(x)-\ell_\beta\rfloor \ \mbox{for all}\ i \ge 1.
\end{displaymath}
This gives the infinite word $d_{-\beta}(x)=x_1x_2x_3\cdots \in \mathcal{A}^\mathbb{N}$ with $\mathcal{A} = \{0,1,\ldots,\lfloor\beta\rfloor\}$.
Since the base is negative, we can represent any $x \in \mathbb{R}$ without the need of a minus sign.
Indeed, let $k\in\mathbb{N}$ be minimal such that $\frac{x}{(-\beta)^k}\in(\ell_\beta,\ell_\beta+1)$ and
$d_{-\beta}(\frac{x}{(-\beta)^k})=x_1x_2x_3\cdots$.
Then the $(-\beta)$-expansion of $x$ is defined as
\begin{displaymath}
\langle x\rangle_{-\beta} = \begin{cases}
x_1\cdots x_{k-1}x_k\bullet x_{k+1}x_{k+2}\cdots & \text{ if $k\geq 1$,}\\
0\bullet x_1x_2x_3\cdots & \text{ if $k=0$.}
\end{cases}
\end{displaymath}

Similarly to positive base numeration systems, the set of \emph{$(-\beta)$-integers} can be defined using the notion of $\langle x\rangle_{-\beta}$, by
\[
\mathbb{Z}_{-\beta} = \{x\in\mathbb{R} : \langle x\rangle_{-\beta} = x_1\cdots x_{k-1}x_k\bullet 0^\omega\} = \bigcup_{k\geq 0}(-\beta)^k T_{-\beta}^{-k}(0)\,,
\]
where $0^\omega$ is the infinite repetition of zeros.
The set of numbers with finite $(-\beta)$-expansion~is
\[
\operatorname{Fin}(-\beta) = \{x\in\mathbb{R} : \langle x\rangle_{-\beta} = x_1\cdots x_{k-1}x_k\bullet x_{k+1}\cdots x_{k+n}0^\omega\}\ =\ \bigcup_{n\geq 0}\frac{\mathbb{Z}_{-\beta}}{(-\beta)^n}\,.
\]
If $\langle x\rangle_{-\beta} = x_1\cdots x_{k-1}x_k\bullet x_{k+1}\cdots x_{k+n}0^\omega$ with $x_{k+n} \ne 0$, then $\operatorname{fr}(x) = n$ denotes the length of the \emph{fractional part} of~$x$; if $x \in \mathbb{Z}_{-\beta}$, then $\operatorname{fr}(x) = 0$.

\section{Finiteness}\label{sec:finiteness--beta}
In this section, we discuss the Property~($-$F) for several classes of Pisot numbers~$\beta$.
Note that $\operatorname{Fin}(-\beta)$ is a subset of $\mathbb{Z}[\beta^{-1}]$ since $\beta$ is an algebraic integer, hence Property~($-$F) means that $\operatorname{Fin}(-\beta) = \mathbb{Z}[\beta^{-1}]$, i.e., $\operatorname{Fin}(-\beta)$ is a ring.
We start by showing that bases~$\beta$ satisfying $d_{-\beta}(\ell_\beta)=d_1d_2\dots d_k0^\omega$, which can be considered as analogs to simple Parry numbers, do not possess Property~($-$F).
This was conjectured in~\cite{PurelyPeriodic} and supported by the fact that $d_{-\beta}(\ell_\beta)=d_1d_2\dots d_k0^\omega$ with $d_1\ge d_j+2$ for all $2 \le j \le k$ implies that $d_{-\beta}(\beta{-}1{-}d_1)=(d_2{+}1)(d_3{+}1)\cdots(d_k{+}1)1^\omega$.
However, the assumption $d_1 \ge d_j +2$ is not necessary for showing that Property ($-$F) does not hold.

We also prove that a base with Property~($-$F) cannot be the root of a polynomial of the form $a_0x^d+a_1x^{d-1}+\dots+a_d$ with $|\sum_{i=0}^d (-1)^i a_i|=1$.

\begin{proof}[Proof of Theorem~\ref{thm:notf}]
If $T_{-\beta}^k(\ell_\beta) = 0$, i.e., $d_{-\beta}(\ell_\beta)=d_1d_2\dots d_k0^\omega$, then we have
\begin{displaymath}
\frac{-\beta}{\beta+1}=\frac{d_1}{-\beta}+\frac{d_2}{(-\beta)^2}+\dots+\frac{d_k}{(-\beta)^k}
\end{displaymath}
and thus $\frac{-1}{\beta+1} \in \mathbb{Z}[\beta^{-1}]$.
However, we have $\frac{-1}{\beta+1} \notin \operatorname{Fin}(-\beta)$ since $T_{-\beta}(\frac{-1}{\beta+1}) = \frac{-1}{\beta+1}$, i.e., $d_{-\beta}(\frac{-1}{\beta+1}) = 1^\omega$.
Hence $\beta$ does not possess Property~($-$F).

If $p(\beta) = 0$ with $|p(-1)| = 1$, then write
\[
p(x-1) = x f(x) + p(-1),
\]
with $f(x) \in \mathbb{Z}[x]$.
Then we have $\frac{1}{\beta+1} = |f(\beta+1)| \in \mathbb{Z}[\beta]$ and thus $\frac{-(-\beta)^{-j}}{\beta+1} \in \mathbb{Z}[\beta^{-1}]$ for some $j \ge 0$.
Now, $d_{-\beta}(\frac{-(\beta)^{-j}}{\beta+1}) = 0^j\, 1^\omega$ implies that $\beta$ does not have the Property~($-$F).
\end{proof}

The main tool we will be using in the rest of the paper are $\alpha$-shift radix systems.
An $\alpha$-SRS is a dynamical system acting on $\mathbb{Z}^d$ in the following way.
For $\alpha\in\mathbb{R}$, $\mathbf{r}=(r_0,r_1,\dots,r_{d-1})\in\mathbb{R}^d$, and $\mathbf{z}=(z_0,z_1,\dots, z_{d-1})\in\mathbb{Z}^d$, let $\tau_{\mathbf{r},\alpha}$ be defined as
\[
\tau_{\mathbf{r},\alpha}(z_0,z_1,\dots,z_{d-1}) = (z_1,\dots, z_{d-1},z_d),
\]
where $z_d$ is the unique integer satisfying
\begin{equation}\label{basicproperty}
0\leq r_0z_0+r_1z_1+\dots+r_{d-1}z_{d-1} + z_d +\alpha <1.
\end{equation}
Alternatively, we can say that
\[
\tau_{\mathbf{r},\alpha}(z_0,z_1,\dots,z_{d-1}) = (z_1,\dots, z_{d-1},-\lfloor \mathbf{rz}+\alpha\rfloor),
\]
where $\mathbf{rz}$ stands for the scalar product.

The usefulness of $\alpha$-SRS with $\alpha=0$ for the study of finiteness of $\beta$-expansions was first shown by Hollander in his thesis~\cite{HOL}. His approach was later formalized in~\cite{ABBPT} where the case $\alpha=0$ was extensively studied. The symmetric case with $\alpha=\frac12$ was then studied in~\cite{AkiSch}. Finally, general $\alpha$-SRS were considered by Surer~\cite{Sur}.

We say that $\tau_{\mathbf{r},\alpha}$ has the finiteness property if for each $\mathbf{z}\in\mathbb{Z}^d$ there exists $k\in\mathbb{N}$ such that $\tau^k_{\mathbf{r},\alpha}(\mathbf{z}) =\boldsymbol 0.$ The finiteness property of $\tau_{\mathbf{r},\alpha}$ is closely related to the Property ($-$F), thus it is desirable to study the set
\[
\mathcal D_{d,\alpha}^0 = \{\mathbf{r}\in\mathbb{R}^d : \forall \mathbf{z}\in\mathbb{Z}^d, \exists k, \tau^k_{\mathbf{r},\alpha}(\mathbf{z})=\boldsymbol 0\}.
\]

The following proposition shows the link between $(-\beta)$-expansions and $\alpha$-SRS.

\begin{proposition}\label{p:conjugacy2}
Let $\beta>1$ be an algebraic integer with minimal polynomial $x^d+a_1x^{d-1}+\dots+a_{d-1}x+a_d$.
Set $\alpha = \frac\beta{\beta+1}$ and let $(r_0,r_1,\dots,r_{d-2}) \in \mathbb{R}^{d-1}$ be such that
\begin{gather*}
x^d+ (-1)a_1x^{d-1} + \dots + (-1)^d a_d = (x+\beta) (x^{d-1}+r_{d-2} x^{d-2} + \dots + r_1x + r_0), \\
\mbox{i.e.,} \quad r_i = (-1)^{d-i} \bigg(\frac{a_{d-i}}{\beta} + \dots +  \frac{a_d}{\beta^{i+1}}\bigg)\quad \mbox{for}\ i = 0,1,\dots,d-2.
\end{gather*}
Then $\beta$ has Property ($-$F) if and only if $(r_0,r_1,\dots,r_{d-2})\in\mathcal D_{d-1,\alpha}^0$.
\end{proposition}

\begin{proof}
Let $\mathbf{r} = (r_0,r_1,\dots,r_{d-2})$.
First we show that for $\phi: \mathbf{z}\mapsto \mathbf{rz}-\lfloor \mathbf{rz}+\alpha\rfloor$ the following commutation diagram holds, i.e., the systems $(\tau_{\mathbf{r},\alpha},\mathbb{Z}^{d-1})$ and $(T_{-\beta},\mathbb{Z}[\beta]\cap[\ell_\beta,\ell_\beta+1))$ are conjugated.
\[
\begin{CD}
{\mathbb{Z}^{d-1}} @>{\tau_{r,\alpha}}>> \mathbb{Z}^{d-1}\\
@VV\phi V 		@VV \phi V\\
{\mathbb{Z}[\beta]\cap [\ell_\beta,\ell_\beta+1)} @>{T_{-\beta}}>> 	{\mathbb{Z}[\beta]\cap [\ell_\beta,\ell_\beta+1)}
\end{CD}
\]
Since $r_i = (-1)^{d-i-1} (\beta^{d-i-1}+a_1 \beta^{d-i-2}+ \cdots + a_{d-i-1})$ for $0 \le i \le d-2$, the set $\{r_i:\, 0\leq i<d\}$ with $r_{d-1}=1$ forms a basis of $\mathbb{Z}[\beta]$, hence $\phi$ is a bijection.
Moreover, we have $-\beta r_i = r_{i-1} + c_i$ with $c_i\in\mathbb{Z}$ and $r_{-1} = 0.$
For $\mathbf{z}=(z_0,z_1,\dots,z_{d-2})$, we have $\phi(\mathbf{z}) = \sum_{i=0}^{d-1} r_iz_i$ with $z_{d-1} = -\lfloor \mathbf{rz}+\alpha\rfloor$, thus
\[
T_{-\beta}(\phi(\mathbf{z})) = -\beta\phi(\mathbf{z})+ n = \sum_{i=1}^{d-1} r_{i-1}z_i + n' = \phi(z_1,\dots,z_{d-2},z_{d-1}) = \phi(\tau_{\mathbf{r},\alpha}(\mathbf{z})),
\]
where $n$ and $n'$ are integers; for the third equality, we have used that $T_{-\beta}(\phi(\mathbf{z})) \in [\ell_\beta,\ell_\beta+1)$.

Therefore, we have $\mathbf{r} \in\mathcal D_{d-1,\alpha}^0$ if and only if for each $x \in \mathbb{Z}[\beta]\cap [\ell_\beta,\ell_\beta+1)$ there exists $k \ge 0$ such that $T_{-\beta}^k(x) = 0$.
Since for each $x\in\mathbb{Z}[\beta^{-1}]\cap [\ell_\beta,\ell_\beta+1)$ we have $T^n_{-\beta}(x)\in\mathbb{Z}[\beta]$ for some $n\in\mathbb{N}$, Property~($-$F) is equivalent to $\mathbf{r} \in\mathcal D_{d-1,\alpha}^0$.
\end{proof}

Thus the problem of finiteness of $(-\beta)$-expansions can be interpreted as the problem of finiteness of the corresponding $\alpha$-SRS.
This problem is often decidable by checking the finiteness of $\alpha$-SRS expansions of a certain subset of~$\mathbb{Z}^d$.
A~\emph{set of witnesses} of $\mathbf{r} \in \mathbb{R}^d$ is a set $\mathcal V\subset \mathbb{Z}^d$ that satisfies
\begin{enumerate}
  \item $\pm \mathbf{e}_i\in\mathcal{V}$ where $\mathbf{e}_i$ denotes the standard basis of $\mathbb R^d,$
  \item if $\mathbf{z}\in\mathcal{V}$, then $\tau_{\mathbf{r},0}(\mathbf{z}),-\tau_{\mathbf{r},0}(-\mathbf{z})\in\mathcal{V}$.
\end{enumerate}
The following proposition is due to Surer~\cite{Sur} and Brunotte~\cite{brunotte}.

\begin{proposition}
Let $\alpha\in[0,1)$ and $\mathbf{r}\in\mathbb{R}^d$.
Then $\mathbf{r}\in\mathcal{D}_{d,\alpha}^0$ if and only if there exists a set of witnesses that does not contain nonzero periodic elements of~$\tau_{\mathbf{r},\alpha}$.
\end{proposition}

Sets of witnesses for several classes of $\mathbf{r}\in\mathbb{R}^d$ were derived in~\cite{SRS2}. Exploiting their explicit form, several regions of finiteness can be determined; see in particular \cite[Theorems~3.3--3.5]{SRS2}.
An $\alpha$-SRS analogy of some of those regions was given by Brunotte~\cite{brunotte}.
Brunotte's result, however, is unsuitable for our purposes.
The next proposition gives several regions of finiteness of $\alpha$-SRS.

\begin{proposition} \label{criteria}
Let $\mathbf{r}=(r_0,r_1,\dots,r_{d-1})\in\mathbb R^d$ and $\alpha\in[0,1).$
\begin{enumerate}
  \item If $\sum_{i=0}^{d-1} |r_i|\leq\alpha$ and $\sum_{r_i<0} r_i > \alpha-1$, then $\mathbf{r}\in \mathcal{D}_{d,\alpha}^0.$
  \item If $0 \le r_0\leq r_1\leq\dots \leq r_{d-1}\leq\alpha$, then $\mathbf{r}\in \mathcal{D}_{d,\alpha}^0$.
  \item If $\sum_{i=0}^{d-1} |r_i| \le \alpha$ and $r_i<0$ for exactly one index $i=d-k$, then $\mathbf{r}\in \mathcal{D}_{d,\alpha}^0$ if and only if
\begin{equation} \label{e:rdjk}
\sum_{1\leq j\leq d/k}r_{d-jk}> \alpha-1.
\end{equation}
\end{enumerate}
\end{proposition}

\begin{proof}
\begin{enumerate}
\item The set $\mathcal V = \{-1,0,1\}^d$ is closed under $\tau_{\mathbf{r},0}(\mathbf{z})$ and $-\tau_{\mathbf{r},0}(-\mathbf{z})$, hence it is a set of witnesses.
For any $\mathbf{z}\in\mathcal V$ we have $|\mathbf{rz}|\le \alpha$, thus $\lfloor  \mathbf{rz} + \alpha\rfloor \in\{0,1\}$.
Hence any periodic point of $\tau_{\mathbf{r},\alpha}$ is in $\{0,-1\}^d$.
For $\mathbf{z}\in\{0,-1\}^d$ we have $\mathbf{rz}+\alpha\le -\sum_{r_j<0}r_j+\alpha < 1$.
Therefore $\lfloor \mathbf{rz} +\alpha\rfloor = 0$, so the only period is the trivial one.

\item In this case we take as a set of witnesses the elements of $\{-1,0,1\}^d$ with alternating signs, i.e., $z_iz_j\le 0$ for any pair of indices $i<j$ such that $z_k=0$ for each $i<k<j$.
For any $\mathbf{z}\in\mathcal V$ we have again $|\mathbf{rz}|\le \alpha$, thus $\lfloor  \mathbf{rz} + \alpha\rfloor \in\{0,1\}$ and $\tau_{\mathbf{r},\alpha}(\mathbf{z}) \in \mathcal{V}$.
Therefore, we have $\tau_{\mathbf{r},\alpha}^n(\mathbf{z}) = (-1,0,\dots,0)$ for some $n \ge 0$, hence $\tau_{\mathbf{r},\alpha}^{n+1}(\mathbf{z}) = \boldsymbol 0$.

\item In this case we have $\mathcal V = \{-1,0,1\}^d$.
As above, all periodic points of $\tau_{\mathbf{r},\alpha}$ are in $\{0,-1\}^d$.
If $\mathbf{z} = (z_0,z_1,\dots,z_{d-1})$ is a periodic point with $z_d = - \lfloor \mathbf{rz} + \alpha\rfloor = -1$, then we must have $z_{d-k} = -1$, and consequently $z_{d-jk}= -1$ for all $1\leq j\leq d/k$.
Then $z_d = -1$ also implies that $-\sum_{1\leq j\leq d/k}r_{d-jk}+\alpha \ge 1$, i.e., \eqref{e:rdjk} does not hold.
On the other hand, if \eqref{e:rdjk} holds, then the vector $(z_0,z_1,\ldots,z_{d-1})$ with $z_{d-jk}= -1$ for $1\leq j\leq d/k$, $z_i=0$ otherwise, is a periodic point of~$\tau_{\mathbf{r},\alpha}$. \qedhere
\end{enumerate}
\end{proof}

Next we prove Property~($-$F) when $\beta$ is a root of a polynomial with alternating coefficients, where the second highest coefficient is dominant.

\begin{proof}[Proof of Theorem~\ref{t:dominant}]
Let $\beta>1$ be a root of $p(x) = x^d-a_{1}x^{d-1}+a_2x^{d-2}+\dots + (-1)^d a_d\in\mathbb{Z}[x]$ with $ a_i\ge 0$ for $i=1,\dots,d$, and $a_{1} \ge 2+\sum_{i=2}^{d} a_{i}$.
As $\frac{\mathrm{d}}{\mathrm{d}x} (p(x)x^{-d}) \ge \frac{a_1}{x^2} - \frac{a_1-2}{x^3} > 0$ for $x > 1$, the polynomial $p(x)$ has a unique root $\beta > 1$, and we have $\beta > a_1-1$ since $p(a_1-1) \le -(a_1-1)^{d-1} + (a_1-2)(a_1-1)^{d-2} < 0$.
By Proposition~\ref{p:conjugacy2}, Property ($-$F) holds if and only if $(r_0,r_1,\dots,r_{d-2})\in\mathcal D_{d-1,\alpha}^0$, with $r_i = a_{d-i} \beta^{-1}-a_{d-i+1}\beta^{-2}+a_{d-i+2}\beta^{-3}-\dots+(-1)^{d-i}a_d\beta^{-d+i-1}$.
We have
\[
- \sum_{r_i<0} r_i \le \frac{a_1-2}{\beta^2} + \frac{a_1-2}{\beta^4} + \dots + \frac{a_1-2}{\beta^{2\lceil d/2\rceil-2}} \le \frac{a_1-2}{\beta^2-1} < \frac{1}{\beta+1}
\]
and
\begin{align*}
& \frac{\beta+1}{\beta} \sum_{i=0}^{d-1} |r_i| \le \frac{\beta+1}{\beta} \bigg(\frac{a_2+\dots+a_d}{\beta} + \frac{a_3+\dots+a_d}{\beta^2} + \dots + \frac{a_d}{\beta^{d-1}}\bigg) \\
& = \frac{a_2+\dots+a_d}{\beta} + \frac{a_2+2a_3+\dots+2a_d}{\beta^2} + \frac{a_3+2a_4+\dots+2a_d}{\beta^3} + \dots + \frac{a_{d-1}+2a_d}{\beta^{d-1}} + \frac{a_d}{\beta^{d}} \\
& \le \frac{a_1-2}{\beta} + \frac{2(a_1-2)-a_2}{\beta^2} + \frac{2(a_1-a_2-2)-a_3}{\beta^3} + \dots + \frac{2(a_1-a_2-\dots-a_{d-1}-2)-a_d}{\beta^d} \\
& \le 1 - 2 \bigg(\frac{1}{\beta} - \frac{a_1-2}{\beta^2} - \frac{a_1-a_2-2}{\beta^3} - \dots - \frac{a_1-a_2-\dots-a_{d-1}-2}{\beta^d}\bigg) \\
& \le 1 - \frac{2}{\beta} \bigg(1 - \frac{a_1-2}{\beta-1}\bigg) < 1.
\end{align*}
Therefore, item~1 of Proposition~\ref{criteria} gives that Property ($-$F) holds.
\end{proof}

Now we can classify the cubic Pisot units with Property ($-$F).
The following description of cubic Pisot numbers in terms of the coefficients of the minimal polynomial is due to Akiyama~\cite[Lemma~1]{AKI}.

\begin{lemma}\label{lem:pisot}
A~number $\beta > 1$ with minimal polynomial $x^3-ax^2+bx-c$ is Pisot if and only if
\[
|b+1|<a+c\quad\text{ and }\quad b+c^2 < \mathrm{sgn}(c)(1+ac).
\]
\end{lemma}

\begin{proof}[Proof of Theorem~\ref{t:cubic}]
Let $\beta>1$ be a cubic Pisot unit with minimal polynomial $x^3-ax^2+bx-c$.
If $c=-1$, then $\beta$ has a negative conjugate, which contradicts Property ($-$F) by~\cite{MPV11}.
Therefore, we assume in the following that $c=1$.
Then from Lemma~\ref{lem:pisot} we have that $-a-1 \leq b < a$.
By Proposition~\ref{p:conjugacy2}, Property~($-$F) holds if and only if $(r_0,r_1)\in\mathcal{D}_{2,\alpha}^0$, with $(r_0,r_1) = (\frac 1\beta, \frac b\beta-\frac 1{\beta^2})$ and $\alpha = \frac{\beta}{\beta+1}$.
We distinguish five cases for the value of~$b$.
\begin{enumerate}
\item
$b=0$: If $a\geq2$, then we have $|r_0|+|r_1| = \frac{1}{\beta} + \frac{1}{\beta^2} < \alpha$ and $r_0+r_1>0>\alpha-1$, so we apply item~3 of Proposition~\ref{criteria}.
If $a=1$, then we have $T^{-1}_{-\beta}(0) = \{0\}$ as $\beta<\frac{1+\sqrt5}{2}$, thus $\operatorname{Fin}(-\beta) = \{0\}$.
\item
$b=-1$:
If $a\geq 1$, then $r_0+r_1 = -\frac1{\beta^2} > \frac{-1}{\beta+1} = \alpha-1$.
If $a\ge 3$, then we also have $|r_0|+|r_1| < \alpha$ and use item~3 of Proposition~\ref{criteria}.
If $a=2$, then $r_0 \approx 0.39$, $r_1 \approx -0.55$, $\alpha \approx 0.72$,  $\{-1,0,1\}^2$ is a set of witnesses, and Property ($-$F) holds because $\tau_{\mathbf{r},\alpha}$ acts on this set in the following way:
\begin{gather*}
\hspace{-2em}
(-1,1) \mapsto (1,1) \mapsto (1,0) \mapsto (0,-1) \mapsto (-1,-1) \mapsto (-1,0) \mapsto (0,0), \\
(0,1) \mapsto (1,0),\ (1,-1) \mapsto (-1,-1).
\end{gather*}
For $a=1$, we refer to Theorem~\ref{t:mdnacci}, which is proved below.
If $a=0$, then $\beta<\frac{1+\sqrt5}{2}$ and thus $\operatorname{Fin}(-\beta) = \{0\}$.
\item
$1 \le b \le a-2$:
For $b \ge 2$, we have $0 < r_0 < r_1 < \alpha$ and thus $(r_0,r_1) \in \mathcal{D}_{2,\alpha}^0$ by item~2 of Proposition~\ref{criteria}.
If $b=1$, then we can use item~1 of Proposition~\ref{criteria} because $r_0,r_1>0$ and $r_0+r_1<\alpha$.
\item
$1 \le b = a-1$:
We have $\beta = b + \frac{1}{\beta(\beta-1)}$.
For $b \ge 3$, we have $0 < r_0 < \alpha < r_1 < 1$, the set $\{-1,0,1\}^2 \setminus \{(1,1),(-1,-1)\}$ is a set of witnesses, and $\tau_{\mathbf{r},\alpha}$ acts on this set by
\[
(1,0) \mapsto (0,-1) \mapsto (-1,1) \mapsto (1,-1) \mapsto (-1,0) \mapsto (0,0), \quad (0,1) \mapsto (1,-1),
\]
thus Property ($-$F) holds.
If $b=2$, then $0 < r_0 < r_1 <\alpha$ and we can use item~2 of Proposition~\ref{criteria}.
If $b=1$, then $r_0 \approx 0.57$, $r_1 \approx 0.25$, $\alpha \approx 0.64$, thus $\{-1,0,1\}^2$ is a set of witnesses, with
\begin{gather*}
(-1,-1) \mapsto (-1,1) \mapsto (1,0) \mapsto (0,-1) \mapsto (-1,0) \mapsto (0,0), \\
(0,1) \mapsto (1,0),\ (1,1) \mapsto (1,-1) \mapsto (-1,0).
\end{gather*}
\item
$-a-1\le b\le-2$:
We have $-r_0-r_1+\alpha = \frac{-b-1}{\beta} +\frac{1}{\beta^2} + \frac{\beta}{\beta+1} > 1$, thus $\tau_{\mathbf{r},\alpha}(-1,-1)=(-1,-1)$, hence $(r_0,r_1) \notin \mathcal{D}_{2,\alpha}^0$.
\end{enumerate}
Therefore, $\beta$ has Property~($-$F) if and only if $-1\le b<a$, $|a|+|b| \ge 2$.
\end{proof}

Finally, we study generalized $d$-bonacci numbers.

\begin{proof}[Proof of Theorem~\ref{t:mdnacci}]
Let $\beta>1$ be a root of $x^{d}-mx^{d-1}-\dots-mx-m$ with $d,m \in \mathbb{N}$.

If $d=1$ (and $m\ge2$), then $\beta$ is an integer, and Property ($-$F) follows from $\mathbb{Z}_{-\beta}=\mathbb{Z}$; see e.g.~\cite{MPV11}.

If $d=3$, then $\mathbf{r}=(\frac m\beta,-\frac m\beta - \frac m{\beta^2})$, $0<r_0<\alpha<-r_1<1$, with $\alpha=\frac{\beta}{\beta+1}$, and $\tau_{\mathbf{r},\alpha}$ satisfies
\[
(0,1) \mapsto (1,1) \mapsto (1,0) \mapsto (0,-1) \mapsto (-1,-1) \mapsto (-1,0) \mapsto (0,0),
\]
with $\{-1,0,1\}^2 \setminus \{(1,-1),(-1,1)\}$ being a set of witnesses.

If $d=5$, then $\mathbf{r} = (\frac m\beta,-\frac m\beta - \frac m{\beta^2}, \frac m\beta + \frac m{\beta^2} + \frac{m}{\beta^3}, -\frac m\beta - \frac m{\beta^2} - \frac m{\beta^3} - \frac m{\beta^4})$, which gives $0<r_0<\alpha<-r_1<r_2<-r_3<1$ and the $\tau_{\mathbf{r},\alpha}$-transitions
\begin{gather*}
(0,1,0,0) \mapsto (1,0,0,1) \mapsto (0,0,1,0) \mapsto (0,1,0,-1) \mapsto (1,0,-1,0) \mapsto (0,-1,0,0) \mapsto \\
(-1,0,0,-1) \!\mapsto\! (0,0,-1,-1) \!\mapsto\! (0,-1,-1,0) \!\mapsto\! (-1,-1,0,0) \!\mapsto\! (-1,0,0,0) \!\mapsto\! (0,0,0,0), \\[1ex]
(0,0,-1,0) \mapsto (0,-1,0,1) \mapsto (-1,0,1,0) \mapsto (0,1,0,-1), \\[1ex]
(0,1,1,1) \mapsto (1,1,1,1) \mapsto (1,1,1,0) \mapsto (1,1,0,-1) \mapsto (1,0,-1,-1) \mapsto \\
(0,-1,-1,-1) \mapsto (-1,-1,-1,-1) \mapsto (-1,-1,-1,0) \mapsto (-1,-1,0,0), \\[1ex]
(0,0,0,1) \!\mapsto\! (0,0,1,1) \!\mapsto\! (0,1,1,0) \!\mapsto\! (1,1,0,0) \!\mapsto\! (1,0,0,0) \!\mapsto\! (0,0,0,-1) \!\mapsto\! (0,0,-1,-1), \\[1ex]
(-1,-1,0,1) \mapsto (-1,0,1,1) \mapsto (0,1,1,0).
\end{gather*}
Let $\mathcal{V}$ be the set of these states.
We have $\pm \mathbf{e}_i \in \mathcal{V}$, $\mathbf{z} \in \mathcal{V}$ if and only if $-\mathbf{z} \in \mathcal{V}$ and $\tau_{\mathbf{r},0}(\mathbf{z}) \in \mathcal{V}$ for all $\mathbf{z} \in \mathcal{V}$, thus $\mathcal{V}$ is a set of witnesses.
As $\tau_{\mathbf{r},\alpha}^{11}(\mathbf{z}) = (0,0,0,0)$ for all $\mathbf{z} \in \mathcal{V}$, $\beta$ has Property~($-$F).

For odd $d \ge 7$, Property~($-$F) does not hold since $T_{-\beta}^{d-1}(\frac{m}{\beta^2} + \frac{m}{\beta^3} + \frac{m}{\beta^4}-1) = \frac{m}{\beta^2} + \frac{m}{\beta^3} + \frac{m}{\beta^4}-1$, i.e., $\tau_{\mathbf{r},\alpha}^{d-1}(-1,0,0,-1,0,0,\dots,0) = (-1,0,0,-1,0,0,\dots,0)$.
For even~$d \ge 2$, we use the second condition of Theorem~\ref{thm:notf}, or that $\tau_{\mathbf{r},\alpha}(-1,\dots,-1) = (-1,\dots,-1)$.

Therefore, $\beta$ has Property ($-$F) if and only if $d\in\{1,3,5\}$.
\end{proof}

\section{Addition and subtraction} \label{sec:addition-subtraction}
In this section, we consider the lengths of fractional parts arising in the addition and subtraction of $(-\beta)$-integers; we prove the following theorem.

\begin{theorem} \label{t:fr}
Let $\beta > 1$ be a root of $x^3 - m \beta^2 - m \beta - m$, $m \ge 1$.
We have
\[
\max\{\operatorname{fr}(x\pm y):\, x, y \in \mathbb{Z}_{-\beta}\} = 3m + \begin{cases} 3 & \mbox{if $m=1$ or $m$ is even}, \\ 4 & \mbox{if $m \ge 3$ is odd}. \end{cases}
\]
\end{theorem}

Throughout the section, let $\beta$ be as in Theorem~\ref{t:fr}, $\mathbf{r} = (r_0,r_1) = (\frac{m}{\beta}, -\frac{m}{\beta}-\frac{m}{\beta^2})$ and $\alpha = \frac{\beta}{\beta+1}$.
Recall that $x, y \in \mathbb{Z}_{-\beta}$ means that $T_{-\beta}^k(\frac{x}{(-\beta)^k}) = 0 = T_{-\beta}^k(\frac{y}{(-\beta)^k})$, and $\operatorname{fr}(x\pm y) = n$ is the minimal $n \ge 0$ such that $T_{-\beta}^{k+n}(\frac{x\pm y}{(-\beta)^k}) = 0$, with $k \ge 0$ such that $\frac{x}{(-\beta)^k}, \frac{y}{(-\beta)^k}, \frac{x\pm y}{(-\beta)^k} \in (\ell_\beta,\ell_\beta+1)$.
To determine $\operatorname{fr}(x-y)$, set
\[
s_j = T_{-\beta}^j(\tfrac{x-y}{(-\beta)^k}) + T_{-\beta}^j(\tfrac{y}{(-\beta)^k}) - T_{-\beta}^j(\tfrac{x}{(-\beta)^k})
\]
for $j \ge 0$.
Then we have $s_j = T_{-\beta}^j(\frac{x-y}{(-\beta)^k})$ for $j \ge k$, and, for all $j \ge 0$,
\begin{gather*}
s_{j+1} \in -\beta s_j + \mathcal{B} \quad \mbox{with} \quad \mathcal{B} = -\mathcal{A}-\mathcal{A}+\mathcal{A} = \{-2m,-2m+1,\dots,m\}, \\[1ex]
s_j \in [\ell_\beta,\ell_\beta+1) + [\ell_\beta,\ell_\beta+1) - [\ell_\beta,\ell_\beta+1) = (\ell_\beta-1,\ell_\beta+2).
\end{gather*}
As $s_0 = 0$, we have $s_j \in \mathbb{Z}[\beta]$ for $j \ge 0$.
Therefore, we extend the bijection $\phi:\, \mathbb{Z}^2 \to \mathbb{Z}[\beta] \cap [\ell_\beta,\ell_\beta+1)$ to
\[
\Phi:\, \mathbb{Z}^2 \times \{-1,0,1\} \to \mathbb{Z}[\beta] \cap [\ell_\beta-1,\ell_\beta+2), \quad (\mathbf{z},h) \mapsto \mathbf{rz} - \lfloor \mathbf{rz} + \alpha \rfloor + h.
\]
Note that $\Phi(\mathbf{z},0) = \phi(\mathbf{z})$.

\begin{lemma}
Let $\mathbf{z} = (z_0,z_1) \in \mathbb{Z}^2$, $h \in \{-1,0,1\}$ and $b \in \mathcal{B}$.
Then
\[
- \beta\, \Phi(\mathbf{z},h) + b = \Phi\big(z_1, h-\lfloor\mathbf{rz}+\alpha\rfloor, \big(z_1-z_0-h+\lfloor\mathbf{rz} +\alpha\rfloor\big)\, m + \big\lfloor r_0z_1 + r_1h - r_1 \lfloor\mathbf{rz} +\alpha\rfloor + \alpha\big\rfloor + b\big).
\]
\end{lemma}

\begin{proof}
We have
\begin{align*}
-\beta\, \Phi(\mathbf{z},h) + b & = -z_0 m + z_1 m + \frac{z_1m}{\beta} + \lfloor\mathbf{rz}+\alpha\rfloor \beta - h \beta + b\\
& = r_0z_1 + r_1\big(h-\lfloor\mathbf{rz}+\alpha\rfloor\big) + \big(z_1-z_0-h+\lfloor\mathbf{rz}+\alpha\rfloor\big)\, m + b.
\end{align*}
\end{proof}

Hence, we have $s_j \in \Phi(\tilde{\tau}_{\mathbf{r},\alpha}^j(\mathbf{0},0))$, where $\tilde{\tau}_{\mathbf{r},\alpha}$ extends $\tau_{\mathbf{r},\alpha}$ to a set-valued function by
\begin{multline*}
\tilde{\tau}_{\mathbf{r},\alpha}:\, \mathbb{Z}^2 \times \{-1,0,1\} \to \mathcal{P}(\mathbb{Z}^2 \times \{-1,0,1\}), \quad (\mathbf{z},h) \mapsto \big\{(z_1,h-\lfloor\mathbf{rz} +\alpha\rfloor,h'): \\
h' \in \{-1,0,1\} \cap
\big(\big(z_1-z_0-h+\lfloor\mathbf{rz} +\alpha\rfloor\big)\, m + \big\lfloor r_0z_1 + r_1h - r_1 \lfloor\mathbf{rz} +\alpha\rfloor +\alpha\big\rfloor + \mathcal{B}\big)\big\}.
\end{multline*}
To give a bound for the sets $\tilde{\tau}_{\mathbf{r},\alpha}^j(\mathbf{0},0)$, let
\begin{gather*}
A_k = \{(j,k):-1\le j<k\},\ B_k = \{(k,j):1\le j\le k\},\ C_k = \{(j,j-k):1\le j\le k\}, \\
D_k = \{(-j,-k):0\le j<k\},\ E_k = \{(-k,-j):2\le j\le k\}, \\
F_k = \{(-j,k-j):2\le j\le k+1\}.
\end{gather*}
Then $\bigcup_{k\ge0} \{A_k,B_k,C_k,D_k,E_k,F_k\}$ forms a partition of $\mathbb{Z}^2 \setminus \{(0,0),(-1,-1)\}$, with the sets $B_0$, $C_0$, $D_0$, $E_0$, $F_0$, and $E_1$ being empty, see Figure~\ref{f:V}.
If $m \ge 2$, then let
\begin{align*}
V &= \bigg( \bigcup_{0\le k\le m} \big( A_k \cup B_k \cup C_k \cup D_k \cup E_k \cup F_k \big) \times \{-1,0,1\} \bigg) \setminus \{(-1,m,1), (0,m,1)\} \\
& \cup \Big(\big(C_{m+1} \setminus \{(m+1,0)\}\big) \times \{1\}\Big) \cup \Big(D_{m+1} \times \{1\}\Big) \cup \Big(\big(D_{m+1} \setminus \{(0,-m-1)\}\big) \times \{0\}\Big)  \\
& \cup \big(D_{m+1} \setminus \{(0,-m-1),(-1,-m-1),(-2,-m-1)\}\big) \times \{-1\} \\
& \cup \Big(\big(\{(0,0),(-1,-1)\} \cup E_{m+1}\big) \setminus \{(-m-1,-m-1)\}\Big) \times \{-1,0,1\} \\
& \cup \Big( \big( F_{m+1} \setminus \{(-m-2,-1), (-m-1,0)\} \big) \times \{-1,0\} \Big).
\end{align*}
If $m = 1$, then we add the point $(-2,0,-1)$ to this set, i.e.,
\begin{multline*}
V = \big(\big\{(0,0), (1,1), (1,0), (0,-1), (-1,-1), (-1,0), (-2,-1)\big\} \times \{-1,0,1\}\big) \\
\cup \big(\big\{(-1,1), (0,1)\big\} \times \{-1,0\}\big) \cup \big(\{(-1,-2)\} \times \{0,1\}\big) \cup \big\{ (1,-1,1), (0,-2,1), (-2,0,-1)\big\}.
\end{multline*}
We call a point $\mathbf{z} \in \mathbb{Z}^2$ \emph{full} if $\{\mathbf{z}\} \times \{-1,0,1\} \subset V$.

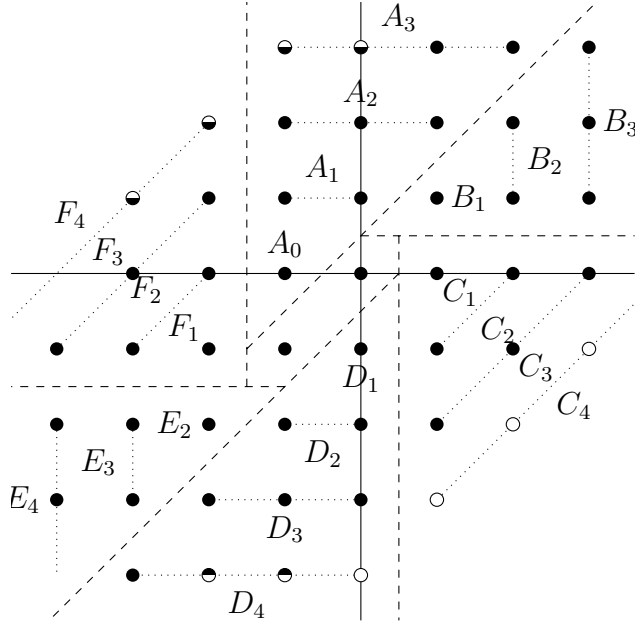
\begin{figure}[ht]
\centerline{\begin{tikzpicture}
\clip(-4.6,-4.6)rectangle(3.6,3.6);
\draw(-4.6,0)--(3.6,0) (0,-4.6)--(0,3.6);
\draw[dashed](-1.5,-1.5)--(-1.5,3.6) (-1.5,-1)--(3.1,3.6) (0,.5)--(3.6,.5) (.5,.5)--(.5,-4.6) (.5,0)--(-4.1,-4.6) (-1,-1.5)--(-4.6,-1.5);
\foreach \k in {0,1,2,3} \draw[dotted](-1,\k)--node[above=2pt]{$A_\k$}(\k-1,\k);
\foreach \k in {1,2,3} \draw[dotted](\k,1)--node[right=1pt]{$B_\k$}(\k,\k);
\foreach \k in {1,2,3,4} \draw[dotted](\k,0)--node[below right=-2pt]{$C_\k$}(1,1-\k);
\foreach \k in {1,2,3,4} \draw[dotted](0,-\k)--node[below=2pt]{$D_\k$}(1-\k,-\k);
\foreach \k in {2,3,4} \draw[dotted](-\k,-2)--node[left=2pt]{$E_\k$}(-\k,-\k);
\foreach \k in {1,2,3,4} \draw[dotted](-\k-1,-1)--node[above left=-1pt]{$F_\k$}(-2,\k-2);
\foreach \x/\y in {1/3, 2/3, 3/3, -1/2, 0/2, 1/2, 2/2, 3/2, -2/1, -1/1, 0/1, 1/1, 2/1, 3/1, -3/0, -2/0, -1/0, 0/0, 1/0, 2/0, 3/0, -4/-1, -3/-1, -2/-1, -1/-1, 0/-1, 1/-1, 2/-1, -4/-2, -3/-2, -2/-2, -1/-2, 0/-2, 1/-2, -4/-3, -3/-3, -2/-3, -1/-3, 0/-3, -3/-4} \fill(\x,\y)circle(2.5pt);
\foreach \x/\y in {3/-1, 2/-2, 1/-3, 0/-4} \filldraw[fill=white](\x,\y)circle(2.5pt);
\foreach \x/\y in {-3/1, -2/2, -1/3, 0/3} {\begin{scope}\filldraw[fill=white](\x,\y)circle(2.5pt); \clip(\x,\y)circle(2.5pt); \fill(\x-1,\y)rectangle(\x+1,\y-1);\end{scope}}
\foreach \x/\y in {-2/-4, -1/-4} {\begin{scope}\filldraw[fill=white](\x,\y)circle(2.5pt); \clip(\x,\y)circle(2.5pt); \fill(\x-1,\y)rectangle(\x+1,\y+1);\end{scope}}
\end{tikzpicture}}
\caption{The set $V$ for $m=3$. Full points are represented by disks, points $\mathbf{z}$ with $\{\mathbf{z}\} \times \{0,1\} \subset V$, $\{\mathbf{z}\} \times \{-1,0\} \subset V$ and $\{\mathbf{z}\} \times \{1\} \subset V$ by upper half-disks, lower half-disks and circles respectively.} \label{f:V}
\end{figure}

The following result is the key lemma of this section.

\begin{lemma} \label{l:PhiV}
Let $x,y \in [\ell_\beta,\ell_\beta+1)$ such that  $x-y \in [\ell_\beta,\ell_\beta+1)$.
Then $T_{-\beta}^j(x-y) + T_{-\beta}^j(y) - T_{-\beta}^j(x) \in \Phi(V)$ for all $j \ge 0$.
\end{lemma}

To prove Lemma~\ref{l:PhiV}, we first determine the value of $\lfloor \mathbf{rz}+\alpha\rfloor$ for $(\mathbf{z},h) \in V$.

\begin{lemma} \label{l:Psi}
Let $\mathbf{z} = (z_0,z_1) \in \mathbb{Z}^2$ with $-m-1 \le z_0 \le m$, $|z_1| \le m+1$ and $|z_0-z_1| \le m+1$.
Then
\[
\lfloor \mathbf{rz}+\alpha\rfloor = z_0 - z_1 + \begin{cases}0 & \mbox{if}\ z_0 \ge 0\ \mbox{or}\ z_1 \le z_0 = -1,\\ 1 & \mbox{if}\ z_0 \le -2\ \mbox{or}\ z_1 > z_0 = -1.\end{cases}
\]
\end{lemma}

\begin{proof}
 We have $z_0 r_0 + z_1 r_1 = z_0 - z_1 - z_0 \frac{m}{\beta^2} + (z_1-z_0) \frac{m}{\beta^3}$ and
\[
\frac{-\beta}{\beta+1} < - \frac{m^2}{\beta^2} - \frac{(m+1)m}{\beta^3} \le - \frac{z_0m}{\beta^2} + \frac{(z_1-z_0)m}{\beta^3} \le \frac{(m+1)m}{\beta^2} + \frac{(m+1)m}{\beta^3} < 1 + \frac{1}{\beta+1},
\]
thus $\lfloor \mathbf{rz}+\alpha\rfloor \in z_0 - z_1 + \{0,1\}$.

If $z_0 \ge 0$, then we have $- z_0 \frac{m}{\beta^2} + (z_1-z_0) \frac{m}{\beta^3} \le (m+1) \frac{m}{\beta^3} < \frac{1}{\beta+1}$.
If $z_1 \le z_0 = -1$, then $- z_0 \frac{m}{\beta^2} + (z_1-z_0) \frac{m}{\beta^3} \le \frac{m}{\beta^2} < \frac{1}{\beta+1}$.
This shows that $\lfloor \mathbf{rz}+\alpha\rfloor = z_0 - z_1$ in these two cases.

If $z_1 > z_0 = -1$, then we have $- z_0 \frac{m}{\beta^2} + (z_1-z_0) \frac{m}{\beta^3} \ge \frac{m}{\beta^2} + \frac{m}{\beta^3} > \frac{1}{\beta+1}$.
Finally, if $z_0 \le -2$, then $- z_0 \frac{m}{\beta^2} + (z_1-z_0) \frac{m}{\beta^3} \ge 2 \frac{m}{\beta^2} - (m-1) \frac{m}{\beta^3} > \frac{1}{\beta+1}$, thus $\lfloor \mathbf{rz}+\alpha\rfloor = z_0 - z_1 + 1$ in the latter cases.
\end{proof}

\begin{proof}[Proof of Lemma~\ref{l:PhiV}]
We have already seen above that $T_{-\beta}^j(x-y) + T_{-\beta}^j(y) - T_{-\beta}^j(x) \in \Phi(\tilde{\tau}_{\mathbf{r},\alpha}^j(\mathbf{0},0))$.
As $(\mathbf{0},0) \in V$, it suffices to show that $\tilde{\tau}_{\mathbf{r},\alpha}(V) \subseteq V$.

Let $(\mathbf{z},h) \in V$ and $b \in B$ such that
\[
h' = (z_1-z_0+\lfloor\mathbf{rz} +\alpha\rfloor-h\big)\, m + \big\lfloor r_0z_1 + r_1h - r_1 \lfloor\mathbf{rz} +\alpha\rfloor +\alpha\big\rfloor + b \in \{-1,0,1\},
\]
i.e., $(z_1,h-\lfloor \mathbf{rz}+\alpha\rfloor,h') \in \tilde{\tau}_{\mathbf{r},\alpha}(\mathbf{z},h)$.
If $(z_1,h-\lfloor \mathbf{rz}+\alpha\rfloor)$ is full, then we clearly have $\tilde{\tau}_{\mathbf{r},\alpha}(\mathbf{z},h) \subset V$.
Otherwise, we have to consider the possible values of~$h'$.
We distinguish seven cases.

\begin{enumerate}
\item
$\mathbf{z} \in \{(0,0),(-1,-1)\}$: We have $\lfloor \mathbf{rz}+\alpha\rfloor = 0$.

If $m \ge 2$, then $(0,h)$ and $(-1,h)$ are full since $(0,-1) \in D_1$, $(0,1), (-1,1) \in A_1$, $(-1,0) \in A_0$, and $(0,0)$, $(-1,-1)$ are also full.

If $m = 1$, then $(0,-1)$, $(0,0)$, $(-1,-1)$, $(-1,0)$ are full.
For $h=1$, we have  $h' = -1 + {\lfloor r_0 z_1 + r_1 + \alpha\rfloor} + b = b-2$, thus $h' = 1$.
The points $(z_1,1,-1)$ for $z_1\in\{-1,0\}$ are in~$V$.

\item
$\mathbf{z} = (j,k) \in A_k$:
We have $\lfloor \mathbf{rz}+\alpha\rfloor = -k$ for $j = -1$ and $\lfloor \mathbf{rz}+\alpha\rfloor = j-k$ for $0 \le j < k$.

If $k=0$, then $\mathbf{z} = (-1,0)$, and $(0,h)$ is full for $m \ge 2$.
If $m=1$, then $(0,0)$ and $(0,-1)$ are full, and $h=1$ gives that $h' = \lfloor r_1 + \alpha\rfloor + b = b-1 \in \{-1,0\}$, thus $\tilde{\tau}_{\mathbf{r},\alpha}(\mathbf{z},1) \subset V$.

If $1 \le k < m$, then $(k,h+k)$, $(k,h-j+k)$ lie in $B_k \cup C_k \cup \{(k,k+1)\}$ and are full.

If $k = m$, then we have either $h \in \{-1,0\}$, thus $(m,h+m)$ and $(m,h-j+m)$ lie in the set of full points $B_m \cup C_m$, or $h = 1$ and $1 \le j < m$, in which case $(m,1-j+m) \in B_m$ is also full. (Note that $(-1,m,1),(0,m,1)\notin V$.)

\item $\mathbf{z} = (k,j) \in B_k$:
We have $\lfloor \mathbf{rz}+\alpha\rfloor = k-j$, $1 \le j \le k$.

For $h \in \{0,1\}$, the point $(j,h+j-k)$ is in $C_k$ and $C_{k-1} \cup B_k$ respectively, hence full.
The point $(j,j-k-1) \in C_{k+1}$ is full if $k < m$.
Finally, if $k=m$ and $h=-1$, then
$h' = m + \lfloor r_0 j + r_1 (j-m-1) + \alpha\rfloor + b = 2m+1 + b = 1$, and $(j,j-m-1,1) \in V$.

\item $\mathbf{z} = (j,j-k) \in C_k$:
We have $\lfloor \mathbf{rz}+\alpha\rfloor = k$, $1 \le j \le k$.

The point $(j-k,h-k)$, $h \in \{-1,0,1\}$, is in $D_{k+1}$, $D_k$, and $D_{k-1} \cup E_{k-1} \cup \{(0,0),(-1,-1)\}$ respectively, hence full for all $k < m$, $k \le m$, and $k \le m+1$ respectively.
It remains to consider $h=-1$, $k=m$.
For $1 \le j \le m-3$, the point $(j-m,-m-1)$ is full; we have
\[
h' = m + \lfloor r_0(j-m) - r_1 (m+1) + \alpha\rfloor + b = \begin{cases}2m+1+b = 1 & \mbox{if}\ j=m, \\ 2m+b \in \{0,1\} & \mbox{if}\ j \in \{m-2,m-1\},\end{cases}
\]
$(0,-m-1,1) \in V$, and $\{(j-m,-m-1)\} \times \{0,1\} \subset V$ for $\max(1,m-2) \le j < m$.

\item $\mathbf{z} = (-j,-k) \in D_k$:
We have $\lfloor \mathbf{rz}+\alpha\rfloor = k-j$ if $j \in \{0,1\}$, $\lfloor \mathbf{rz}+\alpha\rfloor = k-j+1$ if $2 \le j < k$.

Let first $k=1$, i.e., $\mathbf{z} = (0,-1)$.
The point $(-1,h-1)$ lies in $D_2 \cup \{(-1,-1)\} \cup A_0$ and is full, except if $m=1$, $h=-1$; in the latter case, we have $h' = 1 + \lfloor -r_0 - 2r_1 + \alpha\rfloor + b = b+2 \in \{0,1\}$, and $\{(-1,-2)\} \times \{0,1\} \in V$.

For $2 \le k \le m$, the points $(-k,h+j-k)$, $j \in \{0,1\}$, and $(-k,h+j-k-1)$, $2 \le j < k$, lie in $\{(-k,-i):\, 0\le i \le k+1\}$, and are full, except for $k = m = 2$, $h=-1$, $j=0$; in the latter case, we have $h' = 2 + \lfloor -2r_0 - 3r_1 + \alpha\rfloor + b = b+4 \in \{0,1\}$, and $\{(-2,-3)\} \times \{0,1\} \in V$.

Finally, for $k=m+1$, we have $j=0$, $h=1$, or $1 \le j \le \min(m,2)$, $h \in \{0,1\}$, or $3 \le j \le m$, $h \in \{-1,0,1\}$, thus the points $(-m-1,h+j-m-1)$, $j \in \{0,1\}$, and $(-m-1,h+j-m-2)$, $2 \le j \le m$, lie in $\{(-m-1,-i):\, \min(m-1,1) \le i \le m\}$ and are full, except for $m=j=h=1$; in the latter case, we have $h' = -1 + \lfloor -2r_0 + \alpha\rfloor + b = b-2 = -1$, and $(-2,0,-1) \in V$.

\item $\mathbf{z} = (-k,-j) \in E_k$:
We have $\lfloor \mathbf{rz}+\alpha\rfloor = j-k+1$, $2 \le j \le k$.

The point $(-j, h-j+k-1) \in F_{k-2} \cup F_{k-1} \cup F_k \cup \{(-k,-2)\}$ is full, except for $k = m+1$, $h = 1$; in the latter case, we have $2 \le j \le m$, $h' = \lfloor -r_0j + r_1(m-j+1) + \alpha\rfloor + b = b-m \in \{-1,0\}$, and $\{(-j,m-j+1)\} \times \{-1,0\} \subset V$.

\item $\mathbf{z} = (-j,k-j) \in F_k$:
We have $\lfloor \mathbf{rz}+\alpha\rfloor = 1-k$, $2 \le j \le k+1$.

If $1 \le k \le m$, then the point $(k-j,h+k-1) \in A_{k-2} \cup A_{k-1} \cup A_k \cup \{(k-2,k-2)\}$ is full, except for $k = m$, $j \in \{m, m+1\}$, $h = 1$; in the latter case, we have $h' = \lfloor r_0(m-j) + r_1m+\alpha \rfloor + b = b-m \in \{-1,0\}$, and $\{(m-j,m)\} \times \{-1,0\} \subset V$.

If $k=m+1$, then $2 \le j \le m$, $h \in \{-1,0\}$, or $m=1$, $j=2$, $h=-1$, and ${(m+1-j,h+m)} \in A_{m-1} \cup A_m \cup \{(m-1,m-1)\}$ is full. \qedhere
\end{enumerate}
\end{proof}

\begin{lemma} \label{l:TV}
For the following chains of sets, $\tau_{\mathbf{r},\alpha}$ maps elements of a set into its successor:
\begin{gather*}
C_k \setminus \{(m+1,0)\} \to D_k \to E_k \to F_{k-1} \quad (3 \le k \le m+1), \\
F_{k+1} \to A_k \to B_k \to C_k \to D_k \quad (1 \le k \le m).
\end{gather*}
On the remaining $\mathbf{z} = (z_0,z_1) \in \mathbb{Z}^2$ with $-m-1 \le z_0 \le m$, $-m-1 \le z_1 \le m$ and $|z_0-z_1| \le m+1$, $\tau_{\mathbf{r},\alpha}$ acts by
\begin{gather*}
(0,-2) \mapsto (-2,-2) \mapsto (-2,-1) \mapsto (-1,0) \mapsto (0,0), \\ (-1,-2) \mapsto (-2,-1), \quad
(0,-1) \mapsto (-1,-1) \mapsto (-1,0).
\end{gather*}
\end{lemma}

\begin{proof}
This is a direct consequence of Lemma~\ref{l:Psi}, except for $(-m-2,-1) \in F_{m+1}$; see also the proof of Lemma~\ref{l:PhiV}.
As $\frac{1}{\beta+1} < \frac{(m+2)m}{\beta^2} + \frac{(m+1)m}{\beta^3} < 1 + \frac{m}{\beta^2} < 1 + \frac{1}{\beta+1}$, the proof of Lemma~\ref{l:Psi} shows that $\tau_{\mathbf{r},\alpha}(-m-2,-1) = (-1,m) \in A_m$.
\end{proof}

\begin{proposition} \label{p:frn}
We have
\[
\max\{\operatorname{fr}(x-y):\, x, y \in \mathbb{Z}_{-\beta}\} = 3m + \begin{cases} 3 & \mbox{if $m=1$ or $m$ is even}, \\ 4 & \mbox{if $m \ge 3$ is odd}. \end{cases}
\]
\end{proposition}

\begin{proof}
Let $k \ge 0$ be such that $\frac{x}{(-\beta)^k}, \frac{y}{(-\beta)^k}, \frac{x-y}{(-\beta)^k} \in (\ell_\beta,\ell_\beta+1)$.
Then $\operatorname{fr}(x-y) = n$ is the minimal $n \ge 0$ such that $T_{-\beta}^{k+n}(\frac{x-y}{(-\beta)^k}) = 0$.
Let $\mathbf{z} \in \mathbb{Z}^2$ be such that $T_{-\beta}^k(\frac{x- y}{(-\beta)^k}) = \phi(\mathbf{z})$.
Then $\operatorname{fr}(x-y)$ is the minimal $n \ge 0$ such that $\tau_{\mathbf{r},\alpha}^n(\mathbf{z}) = 0$, and  we have $(\mathbf{z}, 0) \in V$, i.e.,
\begin{multline*}
\mathbf{z} \in \{(0,0),(-1,-1),(0,-1)\} \cup \bigcup_{0\le k\le m} \big( A_k \cup B_k \cup C_k \cup D_{k+1} \cup E_{k+1} \cup F_{k+1}\big) \\
\setminus \{(0,-m-1), (-m-1,-m-1), (-m-2,-1), (-m-1,0)\}.
\end{multline*}
Therefore, $\operatorname{fr}(x-y)$ is bounded by the maximal length of the path from $\mathbf{z}$ to $(0,0)$ given by Lemma~\ref{l:TV}.

For $1 \le k \le m/2$, the sets $F_{2k+1}$, $A_{2k}$, $B_{2k}$, and~$C_{2k}$ are mapped to~$D_2$ in $6k-2$, $6k-3$, $6k-4$, and $6k-5$ steps respectively.
For $2 \le k \le (m+1)/2$, the sets $D_{2k}$ and~$E_{2k}$ are mapped to~$D_2$ in $6k-6$ and $6k-7$ steps respectively.
The points $(0,-2)$ and $(-1,-2)$ in $D_2$ are mapped to $(0,0)$ in 4 and 3 steps respectively.

Similarly, for $1 \le k \le (m+1)/2$, the sets $F_{2k}$, $A_{2k-1}$, $B_{2k-1}$, and~$C_{2k-1}$ are mapped to $D_1 = \{(0,-1)\}$ in $6k-2$, $6k-3$, $6k-4$, and $6k-5$ steps respectively.
For $1 \le k \le m/2$, the sets $D_{2k+1}$ and~$E_{2k+1}$ are mapped to~$D_1$ in $6k$ and $6k-1$ steps respectively.
Finally, the point $(0,-1)\in D_1$ is mapped to $(0,0)$ in 3~steps.

For even~$m$, the longest path comes thus from~$D_{m+1}$ and has length~$3m+3$.
For odd $m \ge 3$, the longest path comes from $F_{m+1}$ and has length $3m+4$.
For $m = 1$, the longest path comes from~$A_1$ (since $F_2 \times \{0\} \cap V = \emptyset$ in this case) and has length~6.
This proves the upper bound for $\operatorname{fr}(x-y)$.

For $m = 1$, this bound is attained by $x = 1-\beta$, $y = \beta^4 - \beta^3$, since
 $\operatorname{fr}(x-y) = \operatorname{fr}(\frac{1}{\beta^3}-\beta^3) = \operatorname{fr}(\frac{1}{\beta^3}) = 6$.
Assume in the following that $m \ge 2$.
Then the points $(-m,-m-1,0) \in D_{m+1} \times \{0\}$ and ${(-2,m-1,0)} \in F_{m+1} \times \{0\}$ are in $\tilde{\tau}_{\mathbf{r},\alpha}^j(0,0,0)$ for sufficiently large $j$ because they can be
attained from $(0,0,0)$ by transitions
\[
(\mathbf{z},h) \stackrel{b}{\longrightarrow} \big(z_1, h-\lfloor\mathbf{rz}+\alpha\rfloor, (z_1-z_0+\lfloor\mathbf{rz} +\alpha\rfloor-h\big)\, m + \big\lfloor r_0z_1 + r_1h - r_1 \lfloor\mathbf{rz} +\alpha\rfloor +\alpha\big\rfloor + b\big)
\]
with $b \in B$ by the following paths (cf.\ the proof of Lemma~\ref{l:PhiV}):
\begin{gather*}
(0,k,0) \stackrel{-1}{\longrightarrow} (k,k,-1) \stackrel{-m-k-2}{\longrightarrow} (k,-1,-1) \stackrel{-m-k-2}{\longrightarrow} (-1,-k-2,-1), \quad (0 \le k \le m-2) \\
(-1,-k,-1) \stackrel{-m}{\longrightarrow} (-k,-k,1) \stackrel{k}{\longrightarrow} (-k,0,1) \stackrel{k}{\longrightarrow} (0,k,0), \qquad (2 \le k \le m) \\
(0,m-1,0) \stackrel{-1}{\longrightarrow} (m-1,m-1,-1) \stackrel{-2m}{\longrightarrow} (m-1,-1,0) \stackrel{-m}{\longrightarrow} (-1,-m,-1), \\
(0,m,0) \stackrel{0}{\longrightarrow} (m,m,0) \stackrel{-m}{\longrightarrow} (m,0,0) \stackrel{-m-1}{\longrightarrow} (0,-m,-1),  \\
(0,-m,-1) \stackrel{-m-2}{\longrightarrow} (-m,-m-1,0) \stackrel{-1}{\longrightarrow} (-m-1,-2,1) \stackrel{m}{\longrightarrow} (-2,m-1,0).
\end{gather*}
For even $m \ge 2$, these paths correspond to
\begin{multline*}
\operatorname{fr}(0 0 0 0 2 2\, 0 0 0 0 4 4\, \cdots\, 0 0 0 0 m m\, 0 0 0 0 \bullet 0^\omega - 0 2 2 0 0 0\, 0 4 4 0 0 0\, \cdots\, 0 m m 0 0 0\, 0 0 1 2 \bullet 0^\omega) \\
= \operatorname{fr}((1 m m m 0 0)^{m/2}\, 0 m m m\bullet d_{-\beta}(\phi(-m,-m-1)))
= \operatorname{fr}(\phi(-m,-m-1)) = 3m+3;
\end{multline*}
for the second equality, we have used that $(1 m m m 0 0)^{m/2}\, 0 m m m \bullet  {d_{-\beta}(\phi(-m,-m-1))}$ is a $(-\beta)$-expansion.
Indeed, this follows from the lexicographic conditions given in~\cite{IS09} since $d_{-\beta}(\ell_\beta) = m0m^\omega$ and $d_{-\beta}(\phi(-m,-m-1))$ starts with~$2$ (as $-\beta \phi(-m,-m-1) - 2 = \phi(-m-1,-2)$).
For odd $m \ge 3$, we have
\begin{multline*}
\operatorname{fr}(0 0 0 0 2 2\, 0 0 0 0 4 4\, \cdots\, 0 0 0 0 (m{-}1) (m{-1})\, 0 0 0 0 m m\, 0 0 0 0 0 m \bullet 0^\omega \\
- 0 2 2 0 0 0\, 0 4 4 0 0 0\, \cdots\, 0 (m{-}1) (m{-}1) 0 0 0\, 0 m 0 0 0 0\, 0 0 1 2 0 0 \bullet 0^\omega) \\
= \operatorname{fr}((1 m m m 0 0)^{(m+1)/2}\, 0 m m m 1 0\bullet d_{-\beta}(\phi(-2,m-1)))
= \operatorname{fr}(\phi(-2,m-1)) = 3m+4.
\end{multline*}
This concludes the proof of the proposition.
\end{proof}

\begin{proposition} \label{p:frp}
We have
\[
\max\{\operatorname{fr}(x+y):\, x, y \in \mathbb{Z}_{-\beta}\}\leq \max\{\operatorname{fr}(x-y):\, x, y \in \mathbb{Z}_{-\beta}\}.
\]
\end{proposition}

\begin{proof}
Let $\mu = \max\{\operatorname{fr}(x-y):\, x, y \in \mathbb{Z}_{-\beta}\}$.
For $x,y \in \mathbb{Z}_{-\beta}$, $\operatorname{fr}(x+y)$ is the minimal $n \ge 0$ such that $T_{-\beta}^{k+n}(\frac{x+y}{(-\beta)^k}) = 0$, with $k \ge 0$ such that $\frac{x}{(-\beta)^k}, \frac{y}{(-\beta)^k}, \frac{x+y}{(-\beta)^k} \in (\ell_\beta,\ell_\beta+1)$.
By Lemma~\ref{l:PhiV}, we have
\[
T_{-\beta}^j(\tfrac{x}{(-\beta)^k}) + T_{-\beta}^j(\tfrac{y}{(-\beta)^k}) - T_{-\beta}^j(\tfrac{x+y}{(-\beta)^k}) \in \Phi(V)
\]
for all $j \ge 0$, thus $T_{-\beta}^k(\frac{x+y}{(-\beta)^k}) \in -\Phi(V)$.
Therefore, we have $T_{-\beta}^k(\frac{x+y}{(-\beta)^k}) = \phi(\mathbf{z}) = - \Phi(-\mathbf{z}, h)$ for some $\mathbf{z} = (z_0,z_1) \in \mathbb{Z}^2$ and $h \in \{0,1\}$ with $(-\mathbf{z},h) \in V$.

If $(\mathbf{z},0) \in V$, then the proof of Proposition~\ref{p:frn} shows that $\tau_{\mathbf{r},\alpha}^\mu(\mathbf{z}) = \mathbf{0}$, thus $\operatorname{fr}(x+y) \le \mu$.

Assume now that $(\mathbf{z},0) \notin V$.
Then
\[
-\mathbf{z} \in D_{m+1} \cup \big\{(-m-1,-j):\, 1 \le j \le m\big\} \cup \big\{(-j,m-j+1):\, 1 \le j \le m\big\}.
\]
We can exclude $-\mathbf{z} = (-j,m-j+1)$, $1 \le j \le m$, because this would imply $h=0$ and
\[
-\Phi(-\mathbf{z},h) = 1-\frac{jm}{\beta^2} - \frac{(m+1)m}{\beta^3} \ge \frac{m}{\beta} - \frac{m^2}{\beta^3} > \frac{1}{\beta+1}.
\]
This means that $\mathbf{z} \in (A_{m+1} \cup B_{m+1}) \setminus \{(-1,m+1), (m+1,m+1)\}$.
With the notation of Lemma~\ref{l:TV}, we have
\[
A_{m+1} \setminus \{(-1,m+1)\} \to B_{m+1} \setminus \{(m+1,m+1)\} \to C_m,
\]
where we have used Lemma~\ref{l:Psi} and that $\lfloor r_0(m+1)+r_1j+\alpha\rfloor = m-j$ for $1 \le j \le m$, as
\[
-\frac{\beta}{\beta+1} < \frac{m}{\beta} - \frac{m^2}{\beta^2} + \frac{m-m^2}{\beta^3} \le \frac{m}{\beta} - \frac{m^2}{\beta^2} + \frac{(j-m)m}{\beta^3} \le \frac{m}{\beta} - \frac{m^2}{\beta^2} < \frac{1}{\beta+1}.
\]
Hence, the points in $A_{m+1} \setminus \{(-1,m+1)\}$ and $B_{m+1} \setminus \{(m+1,m+1)\}$ are mapped to $(0,0)$ in the same number of steps as those in $A_m$ and~$B_m$ respectively, thus $\operatorname{fr}(x+y) \le \mu$.
\end{proof}

Now, Theorem~\ref{t:fr} is an immediate consequence of Propositions~\ref{p:frn} and~\ref{p:frp}.

\begin{remark}
It is also possible to determine the exact value of $\max\{\operatorname{fr}(x+y):\, x, y \in \mathbb{Z}_{-\beta}\}$ in the same fashion as in the proof of Proposition~\ref{p:frn}; we have
\[
\max\{\operatorname{fr}(x+y):\, x, y \in \mathbb{Z}_{-\beta}\} = 3m + \begin{cases} 1 & \mbox{if}\ m=2, \\ 2 & \mbox{if $m \ge 4$ is even}, \\ 3 & \mbox{if}\ m=1, \\4 & \mbox{if $m \ge 3$ is odd}. \end{cases}
\]
\end{remark}


\begin{thebibliography}{99}
\bibitem{AKI} S. Akiyama, {\it Cubic Pisot units with finite beta expansions}, Algebraic number theory and Diophantine analysis (Graz, 1998), 11--26, de Gruyter, Berlin, 2000.

\bibitem{ABBPT} S. Akiyama, T. Borb\'ely, H. Brunotte, A. Peth\H{o}, J. Thuswaldner, {\it Generalized radix representations and dynamical systems. I.}, Acta Math. Hungar. 108 (2005), no. 3, 207--238.

\bibitem{SRS2} S. Akiyama, H. Brunotte, A. Peth\H{o}, J. Thuswaldner, {\it Generalized radix representations and dynamical systems. II.}, Acta Arith. 121 (2006), no. 1, 21--61.

\bibitem{AkiSch} S. Akiyama, K. Scheicher, {\it Symmetric shift radix systems and finite expansions}, Math. Pannon. 18 (2007), no. 1, 101--124.

\bibitem{AmFrMaPe} P. Ambro\v z, Ch. Frougny, Z. Mas\' akov\'a, E. Pelantov\'a, {\it Arithmetics on number systems with irrational bases}, Bull. Soc. Math. Belg. 10 (2003), 641--659.

\bibitem{BER} J. Bernat, {\it Computation of $L_\oplus$ for several cubic Pisot numbers},  Discrete Math. Theor. Comput. Sci. 9 (2007), no. 2, 175–-193.

\bibitem{brunotte} H. Brunotte, {\it Symmetric CNS trinomials}, INTEGERS 9 (2009), A19, 201--214.

\bibitem{DMV} D. Dombek, Z. Mas\'akov\'a, T. V\'avra, {\it Confluent Parry numbers, their spectra, and integers in positive- and negative-base number systems}, J. Th{\'e}or. Nombres Bordeaux 27 (2015), 745--768.

\bibitem{DK} K. Dajani, Ch. Kalle, {\it Transformations generating negative $\beta$-expansions}, INTEGERS 11B (2011), A5.

\bibitem{DH} S. Dammak, M. Hbaib, {\it Number systems with negative bases}, Acta Math. Hungar. 142 (2014), no. 2, 475--483.

\bibitem{FR} Ch. Frougny, {\it Confluent Linear Numeration Systems}, Theor. Comput. Sci. 106 (1992), no. 2, 183--219.

\bibitem{FRLAI11} Ch. Frougny, A. C. Lai, {\it Negative bases and automata,} Discrete Math. Theor. Comput. Sci. 13 (2011), no. 1, 75--93.

\bibitem{FRSO} Ch. Frougny, B. Solomyak, {\it Finite beta-expansions}, Ergodic Theory Dynam. Systems 12 (1992), no. 4, 713--723.

\bibitem{GR} P. J. Grabner, A. Peth\H o, R. F. Tichy, G. J. W\"oginger, {\it Associativity of recurrence multiplication}, Appl. Math. Letters 7 (1994), no. 4, 85--90.

\bibitem{HePe} T. Hejda, E. Pelantov\'a, {\it Spectral Properties of Cubic Complex Pisot Units}, Math. Comp. 85 (2016), 401--421.

\bibitem{HOL} M. Hollander, {\it Linear numeration systems, finite beta expansions, and discrete spectrum of substitution dynamical systems}, Ph.D. Thesis, University of Washington, 1996.

\bibitem{IS09}
S. Ito and T. Sadahiro, {\it Beta-expansions with negative bases}, INTEGERS 9 (2009), 239--259.

\bibitem{KS14}
P. Kirschenhofer and J. M. Thuswaldner, {\it Shift radix systems---a survey}, Numeration and substitution 2012, 1--59, RIMS K\^oky\^uroku Bessatsu, B46, Res. Inst. Math. Sci. (RIMS), Kyoto, 2014.

\bibitem{PurelyPeriodic} Z. Mas\'akov\'a, E. Pelantov\'a, {\it Purely periodic expansions in systems with negative base}, Acta Math. Hungar. 139 (2013), no. 3, 208--227.

\bibitem{MPV11}
Z. Mas\'akov\'a, E. Pelantov\'a, T. V\'avra, {\it Arithmetics in number systems with negative base}, Theor. Comp. Sci. 412 (2011), no. 8--10, 835--845.

\bibitem{MES} A. Messaoudi, {\it Tribonacci multiplication}, Appl. Math. Lett. 15 (2002), 981--985.

\bibitem{Parry} W. Parry, {\it On the {$\beta$}-expansions of real numbers}, Acta Math. Acad. Sci. Hungar. 11 (1960), 401--416.

\bibitem{REN}
A. R\'enyi, {\it Representations for real numbers and their ergodic properties},
Acta Math. Acad. Sci. Hung. 8 (1957), 477--493.

\bibitem{Sur} P. Surer, {\it $\epsilon$-shift radix systems and radix representations with shifted digit sets}, Publ. Math. Debrecen 74 (2009), no. 1--2, 19--43.
\end{thebibliography}
\end{document}